\documentclass[11pt,letterpaper]{amsart} \usepackage{amssymb,latexsym}
\usepackage{amsmath, amscd}
\usepackage[all]{xy}
\usepackage[latin1]{inputenc}
\def\f2{\mathbb F_2}
\def\R{\mathbb R}
\def\l{\mathcal L}
\newcommand\im{\text{im\,}}
\newcommand\s[1]{S^{#1}}
\newcommand\tr{\rm tr}
\newcommand\eg{e.\,g.\,}
\newcommand\ie{i.\,e.\,}
\renewcommand\mod{\text{ mod\ }}
\renewcommand\stop{\mathcal S^{TOP}}
\def\Z{\mathbb Z}
\newcommand\rp[1]{ RP^{#1}}
\newcommand\rprp[1]{RP^{#1}\vee RP^{#1}}
\newcommand\RPRP{\rprp\infty}
\newcommand\rpsum[1]{RP^{#1}\#\,RP^{#1}}
\newtheorem{thm}{Theorem}[section]
\newtheorem{lm}[thm]{Lemma}
\newtheorem{prop}[thm]{Proposition}
 \theoremstyle{remark}
\newtheorem{ex}[thm]{Example}
\newtheorem{rem}[thm]{Remark}
\DeclareMathOperator{\Top}{Top} \DeclareMathOperator{\Aut}{Aut}

\begin{document}

\title{Free involutions on $\s 1\times \s n$}

\author{Bj\o rn Jahren}
\address{Dept. of Math., University of Oslo, Norway}
\email{bjoernj@math.uio.no}
\thanks{The first author gratefully
acknowledges support from the Institut Mittag-Leffler (Djursholm, Sweden)
while part of this work was done}
\author{S\l awomir Kwasik}
\address{Dept. of Math., Tulane University, New Orleans, LA, USA}
\email{kwasik@math.tulane.edu}
\thanks{The second named author would like to thank the University of Oslo 
for the generous support during his visit, November 2007.}

\begin{abstract} Topological free involutions on 
$\s 1\times\s n$  are classified up to conjugation. 
As a byproduct we obtain a new computation of the group of concordance 
classes of homeomorphisms of the projective space $RP^n$.
\end{abstract}

\maketitle
\centerline{\today}

\section{Introduction and linear examples}

The main purpose of this paper is to give a classification of free topological
involutions on $\s 1\times\s n$ up to topological equivalence, for
$n\geqslant3$. Our original motivation came  from the results in
\cite{JK}, part of which can be thought of as an investigation of
the quotients within one homotopy class of such involutions in the
case $n=3$.  See also \cite{BDK} for the case $n>3$.  \par
It is known that the four standard involutions on $\s 1\times \s 2$ 
(defined below) represent all
involutions up to conjugacy, see \cite[Theorem C]{T}. This result can also
be recovered by the methods of the present paper. For $\s1\times\s1$ 
elementary considerations show that there are only two involutions up to 
conjugacy; in this case the four standard involutions reduce to 
just two. (See also \cite{A}.) Hence we will concentrate on the case 
$n\geqslant3$ in this paper. 
\par

The classification proceeds in three steps.  First we determine the
homotopy types that occur as quotients.  Then we use surgery theory to
classify manifolds in these homotopy types, and finally we show that
this is the same as classifying involutions up to conjugacy.  \par

In the first step we show that there are only four possible
homotopy types of quotients: $\s1\times\s n,\ \s1\tilde\times\s
n,\ \s1\times\rp n$ and $\rp {n+1}\#\rp {n+1}$ --- all realized by
simple, linear involutions.  This is one of the main technical
parts of the paper. The methods belong to standard homotopy
theory, but providing all the necessary details requires a considerable 
amount of work.  \par

The second and third steps form the core of the paper. Each of the four
homotopy types of quotients has to be treated separately.  The
first two are very easy and for $\rp {n+1}\#\rp {n+1}$ we mostly
refer to calculations done elsewhere. Most of our work in this part of the 
paper is therefore devoted to a thorough understanding of the case   
$\s1\times\rp n$. Here are some of the main ingredients:\par
$\bullet$ The group
$\pi_0(\Aut( \s1\times\rp n))$ of homotopy classes of homotopy
equivalences of $\s1\times\rp n$ is computed. In doing so we
 calculate $\pi_1(\Aut(\rp n))$, thereby correcting an
incorrect statement in the literature (\cite{Y}).  The main tool
is the Federer--Schultz spectral sequence.\par
$\bullet$ The
action of  $\pi_0(\Aut( \s1\times\rp n))$ on the topological
structure set $\stop(\s1\times\rp n)$ is completely determined.
This depends in part on the results of Ranicki's paper \cite{R}, whose
recent appearence proved very fortunate for us.\par
$\bullet$ For free actions on simply connected
manifolds, there is always a one--one correspondence between
homeomorphism types of quotients and conjugacy classes of actions,
but in the non--simply connected case this is not necessarily
true. (For counterexamples in the two-dimensional case, see \cite{A}.) 
However, we show that it holds in the situation studied
here. Hence the classification of involutions up to conjugacy is
the same as the classification of quotient manifolds up to homeomorphism.
\smallskip

We close this introduction with some basic observations and examples
of free involutions on $\s1\times\s n$. In particular we determine the
possible fundamental
groups of the quotient manifolds.  In section 2 we discuss the homotopy
types of the quotients and do all the calculations in the case $n$ odd.
The necessary modifications for $n$ even are in an appendix at the end
of the paper (section 6).\par
Sections 3 has the surgery classification of quotients up to homeomorphism,
and in section 4 we show that this is the same as the classification of
involutions up to conjugation. Finally, in section 5, we show how the
computation of the group of concordance classes of homeomorphisms of 
$\rp n$ follows from results in section 3.
\medskip

If $\tau:\s1\times\s n\to \s1\times\s n$ is a free
involution,  we let $Q=\s1\times\s n/\tau$ be the quotient manifold and
$p:\s1\times\s n\to Q$ the projection. $Q$ is
orientable if and only if $\tau$ preserves orientation. The following
are two easy observations:

\begin{lm}\label{lm:euler} The Euler characteristic $\chi(Q)=0$.
\end{lm}

\begin{proof} This is because $\chi(\s1\times\s n)=2\chi(Q)$ is 0.
\end{proof}

\begin{lm}\label{lm:pi1} If $n\geqslant2$, the fundamental group of $Q$ is 
either $\Z$, $\Z\oplus\Z/2$ or the
infinite dihedral group $D_\infty\cong \Z/2*\Z/2$.
\end{lm}

\begin{proof}
By covering space theory $\pi_1(\s1\times\s n)\cong\Z$ is a subgroup of  index
two in $\pi_1(Q)$, hence it is also normal. If $\pi_1(Q)$ is abelian, it is
then either $\Z$ or $\Z\oplus\Z/2$.  If it  is nonabelian, it
has to be generated by a generator $t$ of  $\Z$ and an element $x$ not
in $\Z$, and  we must have $xtx^{-1}=t^{-1}$. It remains to prove that
$x^2=1$.\par Since $x^2$ maps to the identity in $\Z/2$, we have
$x^2=t^m$, for some  integer $m$. But then
$$t^m=x\cdot x^2\cdot x^{-1}= x t^m x^{-1}=(xtx^{-1})^m=t^{-m}\,,$$
hence  $m=0$.
\end{proof}

\begin{ex}\label{ex:lininvol} The simplest involutions have the form 
$\alpha\times\beta$,
a product of involutions on $\s 1$ and $\s n$, at least one of them free.
Examples are $(t,x)\mapsto(-t,x)$ and $(t,x)\mapsto(t,-x)$ and
$(t,x)\mapsto(\bar t,-x)$ with quotients $\s 1\times\s n,\
\s 1\times\rp n$ and $\rp {n+1}\#\,\rp {n+1}$, respectively. Hence all three
fundamental groups are realizable. Notice, however, that the involution
$(t,x)\mapsto(-t,\bar x)$, where $x\mapsto \bar x$ is any
orientation
reversing involution of $\s n$, has quotient equal to the {\em twisted}
 $\s n$--bundle over $\s 1$, denoted $\s 1\tilde\times \s n$,
which also has fundamental group $\Z$.\par
These are what we refer to as the four {\em standard involutions.}
\end{ex}

\begin{rem}\label{rem:hominv} These examples show that the homotopy type of 
the quotient
depends on more that the homotopy class of the involution as a
{\em map}. For example, $(t,x)\mapsto(-t,x)$ and
$(t,x)\mapsto(t,-x)$ are homotopic if $n$ is odd,
$(t,x)\mapsto(t,-x)$ and $(t,x)\mapsto(-t,\bar x)$ are homotopic
if $n$ is even, but the quotients are not homotopy equivalent. For the 
relationship between involutions with homotopy equivalent quotients, see
the remark at the end of section 4.
\end{rem}

In general, we cannot assume that $\tau$ has this simple form. But the induced
homomorphism in homology does.
If $\tau$ is any self-map of $\s 1\times\s n$,
$\tau_*:H_i(\s1\times\s n)\to H_i(\s1\times\s n),\ i=1,n, n+1$ is
multiplication by an integer $d_i$.  Then $d_{n+1}=d_1d_n$ and the Lefschetz
number of $\tau$ is
$$L(\tau)=1-d_1+(-1)^n d_n+ (-1)^{n+1}d_1d_n=(1-d_1)(1+(-1)^nd_n)\,.$$

If $\tau$ is a free involution, it follows that $d_1=1$ or $d_n=(-1)^{n+1}$.
Now recall the transfer homomorphism $\tr: H_*(Q)\to H_*(\s1\times\s n)$.
It has the  property that $p_*\circ\tr$ is multiplication by 2 and
$\tr\circ p_*= 1+\tau_*$.
It follows that up to elements of order two, $H_i(Q)$ is a direct summand of
$H_i(\s1\times\s n)$, and nontrivial --- hence isomorphic to $\Z$ --- if and
only if $d_i=1$.

\begin{ex}\label{ex:neven} In particular, in the case
$\pi_1(Q)=\Z/2*\Z/2$ we must have $d_1=-1$, and hence $d_n=(-1)^{n+1}$.
Thus $Q$ is orientable if and only if $1=d_1d_n= (-1)^n$, \ie $n$ is
{\em even.}
\end{ex}
\smallskip

\section{Homotopy classification.}\label{sec-homclass}

In this section we prove the following\smallskip

\begin{thm}\label{thm:thm-homclass} Let $n\geqslant 2$. For any free involution $\tau$ on $\s 1\times\s n$
the quotient belongs to one of the four homotopy types
$\s 1\times\s n,\ \s 1\tilde\times\s n,\ \s 1\times RP^n$ and
$\rp {n+1}\#\,\rp {n+1}$.
\end{thm}

\begin{proof}
For the proof we consider separately each of the three fundamental groups given
by Lemma \ref{lm:pi1}.\medskip

{\em Case 1}: $\pi_1(Q)\cong \Z$.\smallskip

This is the simplest case.  Let $\rho:Q\to \s 1$ be a map inducing
an isomorphism on $\pi_1$.  Then $\rho$ classifies the universal
covering of $Q$, hence the homotopy fiber is $\s n$.  Thus $Q$ has the
homotopy type of the total space of an $\s n$-fibration over $\s 1$, and
there are only two such homotopy types --- $\s 1\times\s n$ and
$\s 1\tilde\times \s n$.
\medskip

{\em Case 2}: $\pi_1(Q)\cong \Z\oplus \Z/2$.\smallskip

This time let  $\rho:Q\to \s 1$ be a map inducing a {\em surjection} on
$\pi_1$.  Now $\rho$ has homotopy fiber equivalent to the quotient 
$\overline Q$
of a free involution on $\R\times\s n$.\par

\begin{lm}\label{lm:htpn} Let $\tau$ be a free involution on a finite
 dimensional space $X$ which has the homotopy type of $\s n$.
Then the quotient space $Y=X/\tau$ is homotopy equivalent to $\rp n$.
\end{lm}

\begin{rem}\label{rem-htpn} This is a generalization of the well--known
fact that the
quotient of a free involution on $S^n$ has the homotopy type of $\rp n$.
(See \eg \cite{O},\,\cite{W}.) \cite{TW} contains a proof for $n$ odd.
Note that finite dimensionality is
necessary: an example is $\s\infty \times\s n$ with the
antipodal action on the first factor. 
\end{rem}

\noindent{\em Proof of Lemma \ref{lm:htpn}.}
There is a  fibration (up to homotopy) $X\to Y \to \rp \infty$, and since
$X\simeq \s n$, the usual argument gives a Gysin sequence
$$\cdots\to H_k(Y)\to H_k(\rp\infty)\xrightarrow{\mu_k}
H_{k-n-1}(\rp\infty,H_n(X))\to H_{k-1}(Y)\to\cdots$$

The third term may have twisted coefficients, in which case the homology has
a $\Z/2$ in every even and a zero in every odd dimension. (See Appendix.)
$Y$ has finite dimension $m$, say, hence it has no homology in
degrees above $m$.  It follows that $\mu_k$ must be an isomorphism
for $k$ large.  This can only happen if the coefficients are
twisted for $n$ even and untwisted for $n$ odd. It is now easy to see that
$Y$ has the same homology as $\rp n$ and that the map
$Y\to \rp \infty$ induces the ``same'' map in homology as
$\rp n\subset \rp \infty$. In fact, $\mu_k$ is cap product with a
``Thom class'' in $H^{n-1}(\rp\infty,H^n(X))$, hence must be an isomorphism
whenever $k-n-1>0$. \smallskip

Obviously $Y$ and $\rp n$ also have isomorphic {\em homotopy} groups in
all dimensions. By obstruction theory we see that there is a map
$\rp n\xrightarrow{f} Y$, unique up to homotopy on $\rp {n-1}$ and
inducing an isomorphism on $\pi_i,\ i<n$. All we have to do is to show that
we can choose the map on the top cell of $\rp n$ such that
$f_*: \pi_n(\rp n)\to \pi_n(Y)\cong \Z$ is an isomorphism.\smallskip

With chosen generators for the two
$\pi_n$'s, $\ f_*$ is multiplication by an integer, and we want to show that
we can realize $f_*=1$. The different choices of
extensions of $f$ from $\rp {n-1}$ to $\rp n$ are parametrized by
$\pi_n(Y)$, and an element represented by $h:\s n\to Y$
changes $f$ to the composition
$$\rp n\xrightarrow{\bigtriangledown}\rp n\vee\s n\xrightarrow{f\vee h}Y\,,$$
where $\bigtriangledown$ pinches the boundary of an embedded $n$--disk to
a point. This has the effect of adding an {\em even} number to $f_*$, as is
seen from the diagram:
$$ \xymatrix{
\s n \ar[r]\ar[d]& \s n\cup_{\s 0\times *}\s o \times\s n \ar[r]\ar[d]&
X\ar[d] \\
\rp n \ar[r]^-\bigtriangledown& \rp n\vee \s n \ar[r]^-{f\vee h}& Y
}$$

\noindent where the vertical maps are double coverings and the upper
horizontal maps are the unique base-point preserving liftings.
\smallskip

We now finish the proof of Lemma \ref{lm:htpn} by showing that \smallskip

{\em $f_*:\pi_n(\rp n) \to \pi_n(Y)$ is always multiplication by an
{\em odd} number.}\smallskip

In fact, the map $Y\to \rp \infty$ factors (non--uniquely) through
$\rp n$, since all obstructions to this lie in vanishing groups
($H^k(Y,\pi_{k-1}(\s n)$)). The composite $\rp n\to Y\to \rp n$ is an
isomorphism on $\pi_1$, and any such self--map of $\rp n$ induces
multiplication by an odd number on $\pi_n$.\qed
\medskip

It follows that  $Q$ must be of the homotopy type of an $\rp n$--fibration
over $\s 1$.  If $n$ is even, this fibration must be trivial, since
$\pi_0(\Aut(\rp n))=0$, where $\Aut(M)$ denotes the space of
homotopy self--equivalences of $M$.  If $n$ is odd, $\pi_0(\Aut(\rp n))$ has
two elements,
represented by the identity and reflection  in an $\rp {n-1}$. Hence
$Q$ must be of the homotopy type of either $\s 1\times\rp n$ or a twisted
product $\s 1\tilde\times\rp n$. But it is not difficult to
see that the latter does not have a double covering homotopy equivalent to
$\s 1\times \s n$. In fact, the double coverings are classified by the
index two subgroups of $\pi_1(Q)\cong \Z\oplus\Z/2$. There are three such
subgroups --- two isomorphic to $\Z$ and one isomorphic to $\Z\oplus\Z/2$.
The first two are mapped to each other by an isomorphism of $\pi_1(Q)$
induced by a homotopy equivalence of $Q$. This lifts to a homeomorphism of the
double coverings, which then both have to be $\s 1\tilde\times\s n$. The last
subgroup corresponds to the product of the double covering of $\s 1$
and the identity map of $\rp n$.\medskip

{\em Case 3}:  $\pi_1(Q)\cong \Z/2*\Z/2$.\smallskip

Let $\rho: Q\to \rp\infty\vee\rp\infty$ be a $\pi_1$--isomorphism.  The two
copies of $\rp\infty$ determine two generators of $\pi_1(Q)$ and we
 let $g':\s 1\vee\s 1\to Q$ be the wedge of maps representing
these two generators.

$\pi_i(Q)$ is trivial for $1<i<n$, $\Z$ for $i=n$ and $\Z/2$ for $i=n+1$.
Using the obvious cell structure on $\rp n\vee \rp n$ with $\s 1\vee\s 1$
as 1-skeleton, it then follows by obstruction theory that $g'$ extends to a
map $g: \rp{n}\vee \rp{n}\to Q$ which is an isomorphism on $\pi_i$ for
$i<n$. Moreover, $g|\rp{n-1}\vee \rp{n-1}$ is unique up to homotopy and
$\rho\circ g$ is homotopic to the natural inclusion.
\smallskip

$\rp {n+1} \#\,\rp {n+1}$ is obtained from  $\rp n\vee  \rp{n}$ by
adding  an $(n+1)$-cell which is attached by the map
$\s {n}\to \s {n}\vee \s {n}\to \rp{n}\vee \rp{n}$.  Here
the first arrow is the usual pinch map and the second is the wedge of the two
canonical double covers. We want to show that $g$ can be chosen such that it
extends to a homotopy equivalence $\rp{n+1}\vee\rp{n+1}\to Q$.
\smallskip

The extensions of $g$ from $\rp{n-1}\vee\rp{n-1}$  to $\rp{n}\vee \rp{n}$ are
parametrized by $\pi_{n}(Q)\oplus \pi_{n}(Q)\cong \Z\oplus\Z\,,$
\noindent in the sense that this group acts transitively on the set of
homotopy classes of extensions.  To describe the action and the effect of it,
let $g_1,\, g_2$ be the restrictions of $g$ to the two $\rp n$'s.
Then ${g_i}_*:\pi_n(\rp n)\to \pi_n(Q)$ is multiplication by an integer $a_i$.
{\em Case 3}  now follows immediately from the following Lemma:

\begin{lm}\label{lm:case3}
\begin{enumerate}
\item The numbers $a_1$ and $a_2$ are always odd.
\item $g$ can be extended to $\rp {n+1} \#\,\rp {n+1}$ if and only if
   $a_1+a_2=0$, and the extension is a homotopy equivalence if and only if
   $a_1=-a_2=\pm 1$.
\item The action of an element $(k_1,k_2)\in \pi_{n}(Q)\oplus \pi_{n}(Q)$
replaces $(a_1,a_2)$ by $(a_1+2k_1,a_2+2k_2)$
\end{enumerate}
\end{lm}

\noindent{\em Proof of the Lemma.}
The proof of (3) is exactly like a similar statement in the discussion of
Case 2,  applied to each summand.  To prove (2), observe that the extension
exists if and only if the composite map
$$\s {n}\to \s {n}\vee \s {n}\to \rp{n}\vee \rp{n}\xrightarrow{a_1\vee a_2}Q$$

\noindent is trivial. But this composition represents $a_1+a_2$ in $\pi_n(Q)$.
The second part of (2) follows from the observation that both of the inclusions
$\rp n\to\rp{n+1}\#\,\rp{n+1}$ induce isomorphisms on $\pi_n$. Therefore
the condition  $a_1=-a_2=\pm 1$ is equivalent to the statement that
$g_* : \pi_i(\rp {n+1} \#\,\rp {n+1})\to \pi_i(Q)$ is iso for $i\le n$, and
since both universal coverings are $\R\times \s n$, it must be iso for
 all $i$.\par
It remains to prove (1). We start with a cohomology calculation.
$\rho$ is trivially $n$--connected, hence iso on $H_i$, $i<n$. This determines
the remaining homology of $Q$ as follows:

If $n$ is {\em even}, $Q$ is orientable, and the rest is given by Poincar\'e
duality: $H_{n+1}(Q)\cong \Z$ and $H_n(Q)=0$. This means that we cannot
detect anything using $H_n$.\par
If $n$ is {\em odd}, however, $Q$ is nonorientable; hence $H_{n+1}(Q)=0$ and
$H_n(Q)\cong\Z/2\,\,\oplus$ a free summand.  But the rank of this summand
must be 1 since the Eulercharacteristic is 0. Thus
$H_n(Q)\cong \Z\oplus\Z/2$.
Moreover, $\rho_*:H_n(Q)\to H_n(\rp \infty\vee\rp\infty)$ is onto, since the
composition $H_n(\rp n\vee\rp n)\to H_n(Q)\to H_n(\rp \infty\vee\rp\infty)$
is.
\par

%(Note that this means that $Q$ has (abstractly) the same homology groups as
%$\rp{n+1}\#\rp{n+1}$.)
\smallskip

We also immediately get the ring structure on the mod 2 cohomology
$H^*(Q,\f2)$. In fact, it
follows from the calculations above that both the maps in the composition
 $$H^i(\rp\infty\vee\rp\infty,\f2)\to H^i(Q,\f2)\to H^i(\rp n\vee\rp n,\f2)$$

\noindent are isomorphisms for all $i\le n$. This fact and  Poincar\'e
duality
(for the top dimension) means that the cohomology ring is given as
\begin{equation}
H^*(Q,\f2)\cong \f2[x,y]/(xy, x^{n+1}-y^{n+1})\,.
\end{equation}

The elements $x$ and $y$ are images under $\rho_*$ of the generators for 
the mod 2 cohomology rings of the two copies of $\rp\infty$.
\smallskip

We now finish the proof of (1) in the Lemma, end hence also Theorem 2.1.
Assume first $n$ odd.  It turns out that we can get rid of some
extraneous torsion by dualizing, and this doesn't change the map we are
interested in, ${g_i}_*: \pi_n(\rp n)\to \pi_n(Q)$. So, let
$A^*= Hom(A, \Z)$, for an abelian group $A$, and consider the diagram

\begin{equation}
\xymatrix{
\llap{$\Z \cong\,$} \pi_n(Q)^*\ar[r]^{a_i} &
                               \pi_n(\rp n)^*\rlap{$\, \cong\Z$} \\
\llap{$\Z  \cong\,$} H_n(Q)^*\ar[r]^{{({g_i}_*)}^*}\ar[u]^{{h_Q}^*} &
                      H_n(\rp n)^*\rlap{$\, \cong\Z$} \ar[u]_{{h_P}^*}\\
\llap{$\Z \cong\,$} H^n(Q)\ar[r]^{{g_i}^*}\ar[u]^\cong \ar[d]_{\gamma} &
                   H^n(\rp n)\rlap{$\,\cong\Z$} \ar[u]_\cong \ar@{->>}[d]\\
\llap{$\f2^2 \cong$} H^n(Q,\f2)\ar@{->>}[r]^{{g_i}^*} &
                                H^n(\rp n,\f2) \rlap{$\,\cong \f2$}
}
\end{equation}

$h_P$ and $h_Q$ are Hurewicz maps and the bottom vertical maps are reductions
of coefficients. We know that $h_P$ is multiplication by 2, so  (3) will
follow if we can prove\smallskip

(i) ${h_Q}^*$ is multiplication by an even number, and\par
(ii) ${{g_i}^*}\circ\gamma$ is surjective.
\smallskip

To prove (i), we go back to $h_Q:\pi_n(Q)\to H_n(Q)$ and compare with
the Hurewicz map for the universal covering $\R\times\s n\,$:

$$
\xymatrix{
\pi_n(Q)\ar[d]_{h_Q} & \pi_n(\R\times\s n)\ar[l]_\cong\ar[d]^\cong\\
\llap{$\Z\oplus\Z/2\cong\,$} H_n(Q) & H_n(\R\times \s n)\ar[l]_{p_*}
\rlap{$\,\cong\Z$}
}$$

$p_*(1)$ maps to 0 in $H_n(\rp\infty\vee\rp \infty)$
since $\rho\circ p$ factors through the universal covering of
$\rp\infty\vee\rp \infty$, which is contractible. But
$\rho_*: H_n(Q)\to H_n(\rp\infty\vee\rp \infty)$ is surjective, so this means
that $p_*(1)\in \Z\oplus\Z/2$ must have the form $(2k,l)$.  (1) follows
from this.\smallskip

(ii) uses the cohomology calculation (1).  The image of $\gamma$ must
be nontrivial, and there are just three nontrivial elements in
$H^n(Q,\f2)$\,: ${x_1}^n,\ {x_2}^n  $ and ${x_1}^n+ {x_2}^n$.
Since $n$ is odd, ${\rm Sq}^1(x_1)^n= {\rm Sq}^1(x_2)^n\ne 0$ and
${\rm Sq}^1(x_1^n+x_2^n)=0$, Hence the image of  $\gamma$ must contain
$x_1^n+x_2^n$. But  ${g_i}^*(x_1^n+x_2^n)$ is the generator of
$H^*(\rp n,\f2)$ for both indexes $i$.\smallskip

If $n$ is even we have to use homology with local coefficients.
 $\rp\infty\vee\rp\infty$ has the coefficient system which is twisted by the
nontrivial character on both summands $\rp\infty$, and the other spaces
involved come with maps to $\rp\infty\vee\rp\infty$ and hence also with
induced coefficient systems. With these coefficient systems the
homology/cohomology calculations go just as before. For more details,
see the Appendix, section 6.
\end{proof}
\smallskip

\section{Surgery calculations}\label{surgcalc}\label{section:surgery}\smallskip

To find all manifolds in each of the homotopy types given by Theorem
\ref{thm:thm-homclass}
we use surgery theory. In fact, all the surgery groups involved are
completely known, and all the terms in the surgery exact sequences are
easily computable. Moreover, all groups are ``small'' in the sense of 
\cite{FQ}, such that topological surgery also works in dimension four. 
Hence, with modifications which will be explicitly pointed out, all 
the results in  this section will be valid for $n\geqslant3$. 
Note that all the Whitehead groups involved are
trivial, so we will write just $L$ for the surgery groups $L^s=L^h$.\smallskip

Most of this section will consist of a detailed study of the case
$\pi_1(Q)\cong \Z\times\Z/2$.  $\pi_1(Q)\cong \Z$ is easy, and for
the case $\pi_1(Q)\cong D_\infty$ we will refer to \cite{BDK} and \cite{JK}.
\medskip

\subsection{$Q\simeq\s1\times \s n \ \rm{or}\ \s1\tilde\times\s n$}
The relevant part of the surgery exact sequence is
$$
 [\Sigma(Q_+), G/Top\,]\xrightarrow{\theta'} L_{n+2}(\Z,\omega)\to
\mathcal S^{TOP}(Q)\to [Q,G/Top\,]\xrightarrow{\theta} L_{n+1}(\Z,\omega)\,,
$$
where $\omega:\Z\to \{\pm 1\}$ is the orientation character. A case by case
check ($n\equiv 0,1,2,3\mod4$) reveals that $\theta$ is a bijection
and $\theta'$ is a surjection for all $n\geqslant3$.
 Therefore $\mathcal S^{TOP}(Q)$ has only one element, and
$Q$ is unique up to homeomorphism.  \medskip

\subsection{ $Q\simeq \rp{n+1}\#\rp{n+1}$}  This case has been dealt with
in \cite {BDK} and \cite {JK}. Here we only describe the most interesting 
qualitative results.\par 
For $n\equiv 0$ or 1 mod 4 there are only finitely many
homeomorphism types, all of which are obtained as connected sums
of fake projective spaces. If $n\equiv 3\,$ mod 4 there is an
infinite number of distinct such connected sums, and  when  $n\equiv
2$ or 3 mod 4 there are also infinitely many mutually non--homeomorphic 
examples that can
{\em not} be split as connected sums. These manifolds exist because
of the appearence of big  $U\!Nil$--groups in these dimensions. 
 For more precise statements and calculations, see \cite{BDK}.
(Note, however, that surprisingly little is known about the geometry of 
these manifolds, in particular for $n=3$. See Remarks in \cite[p.251]{JK}).
\medskip

\subsection{$Q\simeq\s1\times\rp n$}
 This will take up the rest of this rather long
section.  We will here make essential use of the fact that the topological
surgery sequence can be considered a natural  exact
sequence of abelian groups if we identify $[M^m/\partial M^m,G/Top\,]$
with $h_m(M,\underline{\mathcal L})$, where $\underline{\mathcal L}$ is the
surgery spectrum. (Note that if $M$ is non--orientable, this homology has
twisted coefficients.) The geometric meaning of this algebraic structure is
still quite mysterious, but \cite {R} goes a long way towards explaining the
functoriality --- at least with respect to homotopy equivalences.
 The neutral 
element of $\stop(M)$ is represented by the identity map on $M$.\smallskip

For $\s1\times\rp n$ we get
\begin{multline}\label{eq:surgpn}
\cdots \to[\Sigma(\s1\times\rp n_+),G/Top\,]\xrightarrow{\theta'}
L_{n+2}(\Z \oplus\Z/2,\omega)\to\\
\to\mathcal S^{TOP}(\s1\times\rp n)
\to[\s1\times\rp n,G/Top\,]\xrightarrow{\theta} L_{n+1}(\Z\oplus\Z/2,\omega)
\end{multline}
The orientation character $\omega$ is nontrivial (isomorphism on $\Z/2$)
if $n$ is even, and trivial if $n$ is odd.
The $L$--groups split as
\[L_*(\Z \oplus\Z/2,\omega)\cong L_*(\Z/2,\omega)
\times L_{*-1}(\Z/2,\omega)\,,\]
compatibly with splittings
\[ [\s1\times\rp n,G/Top\,]\cong[\Sigma\rp n_+,G/Top\,]
\times[\rp n,G/Top\,]\] and
\[[\Sigma(\s1\times\rp n)_+,G/Top\,]\cong
[\Sigma^2\rp n_+,G/Top\,]\times[\Sigma\rp n_+,G/Top\,]\]
of the set of normal invariants. In fact, these splittings extend to a
splitting of the whole surgery sequence
 (\ref{eq:surgpn}):

\begin{lm}\label{lm:surgpn}
There is a splitting 
\[\mathcal S^{TOP}(\s1\times\rp n)\approx
\mathcal S^{TOP}(I\times\rp n) \times\mathcal S^{TOP}(\rp n)\,,\] and the 
surgery sequence (\ref{eq:surgpn})
splits as a product of the surgery sequences for $I\times\rp n$ and $\rp n$.
\end{lm}

\begin{rem} When $n=3$, $\stop(\rp n)$ has to be interpreted as in 
\cite{KT} and \cite{JK1}.  Then the lemma and its proof are valid also in 
this dimension.
\end{rem}

\begin{proof}[Proof of Lemma \ref{lm:surgpn}] Note that when $M$ has 
boundary, $\mathcal S^{TOP}(M)$ denotes 
the set of equivalence classes of homotopy structures which are 
homeomorphisms on the boundary.  Using the group structure, we have 
an obvious map
$$\mathcal S^{TOP}(I\times\rp n) \times\mathcal S^{TOP}(\rp n)\to
\mathcal S^{TOP}(\s1\times\rp n)\,,$$
compatible with the isomorphisms above, and the result follows by
the five--lemma.  
\end{proof}

\begin{rem}
A topological version of  Cappell's splitting theorem \cite{C} provides a 
splitting map
$\rho:\mathcal S^{TOP}(\s1\times\rp n)\to \mathcal S^{TOP}(\rp n)$.
\end{rem}

Observe that the surgery sequence for $I\times\rp n$ is part of the sequence
for $\rp n$ (``forget the last three terms''):
\begin{multline}
\cdots[\Sigma^2(\rp n_+),G/Top\,]\xrightarrow{\theta''}
L_{n+2}(\Z/2,\omega)\to \\
\to\mathcal S^{TOP}(I\times\rp n)\to[\Sigma(\rp n_+),G/Top\,]
\xrightarrow{\theta'} L_{n+1}(\Z/2,\omega)\to\\
\to\mathcal S^{TOP}(\rp n)\to [\rp n,G/Top\,]
\xrightarrow{\theta} L_n(\Z/2,\omega)
\end{multline}

The $L$--groups occurring here are as follows (for 
$n\equiv 0, 1, 2, 3\mod4$) \cite{W}\,:
\begin{eqnarray*}
&L_{n}(\Z/2,\omega)&=\ \Z/2,\, 0,\,\Z/2,\, \Z/2\\
&L_{n+1}(\Z/2,\omega)& =\ 0,\, \Z/2,\,0,\,\Z^2\\
&L_{n+2}(\Z/2,\omega)&=\ \Z/2,\,\Z/2,\,\Z/2,\,0
\end{eqnarray*}

The usual cohomology calculation gives
\begin{eqnarray*}
&[\rp n,G/Top\,] &\cong\ (\Z/2)^{[\frac n2]}\\
&[\Sigma(\rp n_+),G/Top\,] &\cong\ (Z/2)^{[\frac{n+3}4]}\times K_n\,,
\end{eqnarray*}
where  $K_n=\Z$ if $n\equiv3\mod4$ and $K_n=0$ otherwise.\par
The surgery obstruction maps $\theta$ and $\theta'$ have been computed, 
\eg in \cite[Ch.\,IV]{dM}.
L\'opez de Medrano treats the $PL$ case, but the results are essentially
the same in the topological case. Note, however, that each $\Z/4$ in his 
Theorem IV.3.4 is to be replaced by $(\Z/2)^2$ in the topological case. 
\smallskip

The result is:
\begin{itemize}
\item $\theta$ is surjective, and
\item $\text{coker}\,\theta'\cong K_n$, and it sits as
a direct summand of $\stop(\rp n)$.
\end{itemize}

More precisely: The $K_n$--factor of $[\Sigma(\rp n_+),G/Top\,]$ (for 
$n\equiv3\mod4$) maps isomorphically onto one of the $\Z$-factors of 
$L_{n+1}(\Z/2,\omega)$, and the cokernel is again isomorphic to $K_n$.\par
Since $[\rp n,G/Top\,]$ is a product of $\Z/2$'s, this determines 
$\mathcal S^{TOP}(\rp n)$ completely. 
Clearly, every element in $\mathcal S^{TOP}(\rp n)$ gives rise to an 
involution on $\s1\times\s n$ of the form  $\rm{id}_{\s1}\times\tau$, 
where $\tau$ is a free involution on $\s n$. \smallskip

We can also determine the last surgery obstruction map, $\theta''$; 
 we claim that it is always surjective\,:\par
When $n$ is even, $L_{n+2}\cong\Z/2$ is detected by the Arf
invariant and can be easily realized by the product formula.\par 
If $n\equiv 3\,\mod4$, $L_{n+2}=0$, so there is nothing to prove.
In the remaining case, $n\equiv 1\,\mod4$, the obvious maps
$\Z/2\cong L_2(1)\xrightarrow{\times\Z}L_3(\Z)\to L_3(\Z/2)$ are all
isomorphisms (\cite[ch.\,13A]{W}).  The non--trivial element can be 
realized on a neighborhood of $\{0\}\times\rp1\subset D^2\times\rp n$.
\smallskip

This also determines $\stop(I\times\rp n)$, and for later use 
(section \ref{section:conc}) we record
\begin{equation}\label{eq:pnxi}
\stop(I\times\rp n)\cong (\Z/2)^{[\frac{n+2}4]}\,.
\end{equation}

(The exponent is $[\frac{n+3}4]$ for $n\not\equiv1\mod 4$
and $[\frac{n+3}4]-1$ for $n\equiv1\mod 4$, and one easily checks that 
both can also be written as $[\frac{n+2}4]$.)\par
We  remark that,  by the $s$--cobordism theorem, an element in
$\ \stop(I\times\rp n)\subset \stop(\s1\times\rp n)$ can be represented
by the mapping torus of a (degree one) homeomorphism $h\!:\!\rp n\to \rp n$,
together  with a homotopy between $h$ and the identity. \medskip

Putting all this together, we see that the exact sequence (\ref{eq:surgpn}) 
is reduced to a sequence of the form
\begin{equation}\label{eq:surgpnred}
0\to K_n\to\stop(\s1\times\rp n)\xrightarrow{\eta}(\Z/2)^{[\frac n2]}\times
 (Z/2)^{[\frac{n+3}4]}\to\Z/2\to0\,,
\end{equation}
where $K_n$ splits off $\stop(\s1\times\rp n)$ as a direct summand.\smallskip

It follows that
\begin{equation}\label{stop}
\stop(\s1\times\rp n)\cong (\Z/2)^{[\frac n2]+[\frac{n-1}4]}\times K_n\,.
\end{equation}
\smallskip

Our goal is to determine the set of {\em homeomorphism classes}  of manifolds
homotopy equivalent to $\s1\times\rp n$.  This is the quotient of the
structure set $\stop(\s1\times\rp n)$ by the group
$\pi_0(\Aut(\s1\times\rp n))$ of homotopy classes of homotopy equivalences
of $\s1\times\rp n$, acting by
composition. Hence we first need to compute this group.

\begin{prop}\label{prop:autpn} Let $n\geqslant 2$. Then
$$\pi_0(\Aut(\s 1\times\rp n))\cong
\begin{cases} (Z/2)^3 & \text{if $n$ is even}\\
 (Z/2)^3\times\Z/4 & \text{if \ }n\equiv1\mod 4\\
(\Z/2)^5 & \text{if \ }n\equiv3\mod 4
\end{cases}$$

In each case there is a $\Z/2$--summand generated by a tangential
homotopy equivalence with non--trivial normal invariant. The other summands
can be represented by homeomorphisms.
\end{prop}

\medskip

\begin{proof} We begin by observing that there is a decomposition
$$\pi_0(\Aut(\s1\times\rp n))\cong \Z/2 \times\pi_0(\Aut(\rp n ))\times \pi_1(\Aut(\rp n ))$$
The first two factors represent the product maps. Here we first have
$$\pi_0(\Aut(\rp n )) = Z/2\ {\rm for\ }n\ {\rm odd\ and\ } 0\ {\rm for\ }n\
{\rm even.}$$
(The generator is given by the reflection
$r_0([x_0,....x_n])= [ -x_0,x_1,...,x_n]=
[x_0,-x_1,....,-x_n]$,
and if $n$ is even, this can be deformed linearly to the identity.)\par

$\pi_1(\Aut(\rp n ))$ can be computed using the methods in \cite{S}.
(The basepoint in $\Aut(\rp n)$ is the identity map.)
Let $\Aut_{\Z/2}(\s n)$ be the space of homotopy equivalences of $\s n$ 
which are equivariant with respect to the antipodal action
of $\Z/2$.  Taking quotients defines  a double covering projection
$\Aut_{\Z/2}(\s n)\to \Aut(\rp n )$, which is trivial if $n$ is even and
non--trivial (over each component) if $n$ is odd. To see this, observe that
the preimage of the identity map of $\rp n$ consists of the identity and the
antipodal maps of $\s n$.  These are not even homotopic for $n$ even, but 
equivariantly (in fact, linearly) isotopic for $n$ odd.  \par
Thus, for $n$ even there is an isomorphism
$$\pi_1(\Aut_{\Z/2}(\s n))\cong\pi_1( \Aut(\rp n ))\,,$$
whereas for $n$ odd there is an exact sequence
\begin{equation}\label{eq:oddcase}
0\to \pi_1(\Aut_{\Z/2}(\s n))\to\pi_1( \Aut(\rp n ))\to \Z/2\to 0\,.
\end{equation}

 Let $\Gamma_n$ be the space of
sections of the fibration $\rho\,:\,\s n\times_{\Z/2} \s n\to \rp n $
(diagonal action), induced by projection on the first factor. 
Then there is a homeomorphism $\Aut_{\Z/2}(\s n)\approx\Gamma_n$
defined by taking graphs and quotients by the $\Z/2$--actions.
The fibration $\rho$ has simple fibers ($\s n$), so we can apply the 
spectral sequence in \cite{S}.  This  is a
homology type, second quadrant spectral sequence with
$E_{p,q}^2=H^{-p}(\rp n , \pi_q(\s n))$, converging
to $\pi_{q+p}(\Gamma_n)$. \smallskip

The $E^2$--term has local coefficients coming from the monodromy in the
fibration $\rho$. When $q=n$ this is the same as the orientation system of 
$\rp n $ --- hence $E_{p,n}^2\cong H_{n+p}(\rp n ,\Z)$ with {\em trivial}
coefficients. In particular, $E^2_{-n+1,n}\cong\Z/2$ and  $E^2_{-n+2,n}\cong0$.
\par
When $q=n+1$ and $n\geqslant3$, the coefficients are automatically trivial,
since $\pi_{n+1}(\s n)\cong\Z/2$. Therefore  $E^2_{-n,n+1}\cong\Z/2$. 
If $n=2$, the monodromy is induced by the antipodal map on $\s 2$, but this
is trivial on $\pi_3(\s2)\cong\Z$. One way to see this is to use that two times
a generator of $\pi_3(\s2)$ is represented by the Whitehead product of
the identity with itself. Hence the map induced on $\pi_3(\s2)$ by a map of 
degree $d$ on $\s2$  is multiplication by  $d^2$.  It follows that also 
$E_{-2,3}^2=H^2(\rp 2,\Z)\cong\Z/2$.\par 
Thus, in total degree $p+q=1$ we have only two non--trivial groups,
$E^2_{-n,n+1}\cong E^2_{-n+1,n}\cong\Z/2$. There are no differentials 
involving these groups, hence 
also  $E^\infty_{-n,n+1}\cong E^\infty_{1-n,n}\cong Z/2$.
The resulting exact sequence
$$0\to H^n(\rp n ,\pi_{n+1}(\s n))\to \pi_1(\Gamma_n)\to
H^{n-1}(\rp n ,\pi_{n}(\s n))\to 0$$
splits, since the generator of $\pi_1(SO_{n+1})$ defines an element
of order two in $\pi_1(\Aut_{\Z/2}(\s n))$ which maps to the
generator of $H^{n-1}(\rp n ,\pi_{n}(\s n))\cong\Z/2$. (Using the standard 
cell structure of $\rp n$, the image is 
represented by the cellular cocycle mapping the unique  $(n-1)$--cell to 
a homotopy equivalence of $\s n$, by the construction in \cite[p.\,52]{S}.)
It follows that this
factor can be represented by a one--parameter family of homeomorphisms
--- hence the associated homotopy equivalence of $\s1\times\rp n$ can also
be represented by a homeomorphism.\par
The other factor, however, can not. It corresponds to (one--parameter
families of) maps that restrict to the inclusion on $\rp{n-1}$, and the
nontrivial element can be constructed as the composition
$$\s1\times \rp n\to\s1\times \rp n\vee\s{n+1}
\xrightarrow{\rm{pr}\vee\eta}\rp n\,,$$
where $\eta$ is the generator of $\pi_{n+1}(\rp n)$. It follows \eg by the
method of \cite[pp.\,31--32]{KS} that the associated homotopy equivalence
$h_\eta$ of $\s1\times\rp n$ has non--trivial normal invariant; we
only need to observe is that $h_\eta$ is a {\em tangential} homotopy 
equivalence. But this follows since $\eta^*\tau(\rp n)=\eta^*\tau(\s n)$ 
is stably trivial. \smallskip

It remains to examine the exact sequence (\ref{eq:oddcase}). Write $n=2m-1$,
and  think of $\rp n$ as a quotient of the unit sphere in $\mathbb C^m$.
Consider the family
$g_t(z_1,\dots,z_m)=(e^{\pi it} z_1,\dots,e^{\pi it} z_m)$, $t\in[0,1]$.
This is a path of equivariant homeomorphisms of $\s n$, but the path is closed
only in $\Aut(\rp n)$. Hence it maps to the nontrivial element in $\Z/2$ to
the right. \par
The image $\rho$ of the generator of $\pi_1(SO_{n+1})$ can be represented
by $(t,z)\mapsto e^{2\pi it}z,\ t\in[0,1]$ in any one of the coordinates
$z=z_i$. Hence
$ \rho^m=[g]^2$ in $\pi_1(\Aut(\rp n))$.  Since $\rho^2=1$, it follows that
if $m$ is odd, $[g]$ has order 4 and $\rho=[g]^2$, and if $m$ is even, it 
has order 2. Since $\pi_1(\Aut(\rp n))$ is abelian, the result follows.
 \end{proof}

\begin{rem}\label{rem:autpn} (i) Since homeomorphisms are tangential, it 
follows that {\em all}
homotopy self--equivalences of $\s1\times\rp n$ are tangential.\smallskip

(ii) The obvious map $SO(n+1)\to\Aut(\rp n)$ clearly factors through
$PSO(n+1)=SO(n+1)/\text{center}$. Then the main calculation in the above
proof  can be formulated as
$$\pi_1(\Aut(\rp n))\cong \pi_1(PSO(n+1))\times\Z/2\,,$$
--- the last factor being represented by $h_\eta$.\par
 
\end{rem}

To compute the action of $\pi_0(\Aut(\s1\times\rp n))$ on the structure set,
we use the results of \cite{R}.  A homotopy self--equivalence $h$ of a
manifold $M$ induces an automorphism $h_*:\stop(M)\to \stop(M)$ (by
functoriality).  Let $g:N\to M$ represent an element in $\stop(M)$. Then
Ranicki shows that the composition $hg$ represents the element
$[h]+h_*([g])\in\stop(M)$. \medskip

$h_*$ can be computed from the induced diagram of exact sequences
(\ref{eq:surgpn}), which we now know can be written in the following form:
$$ \xymatrix{
0\ar[r]&K_n \ar[r]\ar[d]^{h_*}&\stop(\s1\times\rp n)\ar[r]^\eta\ar[d]^{h_*}&
h_{n+1}(\s1\times\rp n,\underline{\mathcal L})\ar[d]^{h_*} \\
0\ar[r]&K_n \ar[r]& \stop(\s1\times\rp n) \ar[r]^\eta&
h_{n+1}(\s1\times\rp n,\underline{\mathcal L})
}$$

For the next result it will be convenient to introduce the homomorphism
$\mu:\pi_0(\Aut(\s1\times\rp n))\to \{\pm 1\}$ defined for $n$ odd as follows:
\par
When $n$ is odd, $H_n(\s1\times\rp n,\Z)\cong \Z$. Hence, if
$h\in \Aut(\s1\times\rp n)$, the homomorphism induced on $H_n$ by
$h$ is multiplication by an integer in $\{\pm1\}$, which we denote by
$\mu(h)$. $\mu$ is clearly a homomorphism.\par
Let $r_0:\rp n\to \rp n$ be an orientation reversing reflection and put
$r=I_{\s1}\times r_0$. Then $\mu(r)=-1$, and $-1\mapsto r$ defines a
splitting of $\mu$. Note that $\mu$ maps all the other generators of
$\pi_0(\Aut(\s1\times\rp n))$ identified in Prop. \ref{prop:autpn}
to 1.

\begin{lm} $h_*:h_{n+1}(\s1\times\rp n,\underline{\mathcal L})\to
h_{n+1}(\s1\times\rp n,\underline{\mathcal L})$ is the identity for every
homotopy equivalence $h$.  Hence
$h_*:\stop(\s1\times\rp n)\to\stop(\s1\times\rp n)$
is always the identity if $n\not\equiv3\mod4$.\smallskip

If  $n\equiv3\mod4$, then $h_*:K_n\to K_n$ is multiplication by
$\mu(h)$.\smallskip

 This determines $h_*:\stop(\s1\times\rp n)\to\stop(\s1\times\rp n)$
completely.
 \end{lm}

\begin{proof}
It is easy to see that
$$h^*:[\s1\times\rp n,G/Top\,]\to[\s1\times\rp n,G/Top\,]$$
is the identity homomoprhism for every $h$.  For the first part of the Lemma,
it then suffices to observe that via the identification
$$ h_{n+1}(\s1\times\rp n,\underline{\mathcal L})=[\s1\times\rp n,G/Top\,]$$
we have $h_*=(h^*)^{-1}$.  For a general homotopy equivalence
$h:M\to N$, this is true if $h_*$ preserves $L$--theory fundamental
classes, \ie $h_*([M]_L)=[N]_L$.  (See \cite[Lemma 2.5]{R}.) This is not
always the case, but it follows from the characterization of fundamental
classes in \cite[Prop. 16.16]{R2} that it holds for {\em tangential} homotopy
equivalences --- hence for all $h\in \Aut(\s1\times\rp n)$, by  remark
\ref{rem:autpn}.
 \smallskip

 When  $n\equiv3\mod4$, $K_n\cong\Z$.  Hence, since $\im\eta$ is
torsion, $h_*$ is determined by the two components $h_*:K_n\to K_n$ and
$h_*:h_{n+1}(\s1\times\rp n,\underline{\mathcal L})\to
h_{n+1}(\s1\times\rp n,\underline{\mathcal L})$.   $K_n$ comes from the
surgery sequence of $\rp n$, and the homomorphism
$K_n\to \stop(\s1\times\rp n)$ factors as
$$K_n\to \stop(\rp n)\xrightarrow{\s1\times(-)} \stop(\s1\times\rp n)\,.$$

Both of these maps are split injections, and if 
$h:\s1\times\rp n\to\s1\times\rp n$ is a homotopy equivalence, 
$h_*:K_n\to K_n$ is can be described as the composition in the following 
diagram:

$$ \xymatrix{
K_n \ar[r]\ar[d]^{h_*}&\stop(\rp n)\ar@{-->}[d]\ar[r]&
\stop(\s1\times\rp n)\ar[d]^{h_*} \\
K_n & \stop(\rp n)\ar[l]_{BL\ \ \ \ }&\stop(\s1\times\rp n)\ar[l]_{\rho\ \ }
}$$
(The dashed arrow is defined by the diagram.)
The {\em splitting map} $\rho$  maps a class $[f]$  to 
$f^{-1}(t\times\rp n)\xrightarrow{f}\rp n$ for a
suitable representative $f:N\to \s1\times\rp n$ and $t\in\s1$, and $BL$ is
the Browder--Livesay invariant.  From the proof of Prop. \ref{prop:autpn} we
see that every $h$ can be chosen such that $h$ preserves some $t\times\rp n$
and is either the identity or the reflection $r_0$ there, and this is
determined by $\mu(h)$.  Then the vertical map $\stop(\rp n)\to \stop(\rp n)$
is either the identity or ${r_0}_*$, and the result follows \eg by the
discussion preceeding Theorem 4 in \cite{BDK}.
\end{proof}

It follows that for  $n\not\equiv3 \mod4$, the action of the homotopy
equivalence group on $\stop(\s1\times\rp n)$ reduces to translation by the
structure represented by $h_\eta$.
For  $n\equiv3 \mod4$, we also have the homeomorphism $r$ acting
by multiplication by $-1$ (only non--trivial on $K_n\cong\Z$). The set of
homeomorphism classes of manifolds homotopy equivalent to $\s1\times\rp n$
is in one--one correspondence with the elements of the quotient of
 $\stop(\s1\times\rp n)$ under this action, \ie with
\begin{equation}\label{eqn:numbers}
\begin{cases}
(\Z/2)^{[\frac n2]+[\frac{n-1}4]-1} \text{ for }n\not\equiv3 \mod4\\
\mathbb N\times (\Z/2)^{[\frac n2]+[\frac{n-1}4]-1} \text{ for }
n\equiv3 \mod4
\end{cases}
\end{equation}

($\mathbb N$ is the set of natural numbers $\{0,1,2,\dots\}$.) It may be
simpler to think of the exponent $\ell={[\frac n2]+[\frac{n-1}4]-1}$ as 
follows: 
\[\rm{If\ } n=\begin{cases} 4m\\ 4m+1\\ 4m+2\\ 4m+3\end{cases} 
\quad \rm{then\ }\ell=\begin{cases}
3m-2\\ 3m-1\\ 3m\\  3m
\end{cases}\]
\smallskip

\section{Topological classification up to conjugacy}

In the previous section we classified all possible quotients of free
involutions on $\s1\times\s n$ up to homeomorphism.  The following theorem
says that this is the same as the classification of free involutions up to
conjugation.

\begin{thm}
Let $n\geqslant3$. Then two free involutions on $\s1\times\s n $ are conjugate
if and only if the two quotients are homeomorphic.
\end{thm}

(This is also true for $n<3$, by the classification theorems in these cases
\cite{A}\cite{T}. Only the last part of the following proof requires
$n\geqslant3$.)

\begin{proof} One way is trivial: topologically conjugate involutions clearly
have homeomorphic quotients. \par
Let $\tau_1$ and $\tau_2$ be two free involutions, and assume
that $f:Q_{\tau_1}\to Q_{\tau_2}$ is a homeomorphism. Let
$p_i:\s1\times\s n\to Q_{\tau_i},\ i=1,2$ be the two projections. Choose a
basepoint $x_0\in\s1\times\s n$ and let $q_1=p_1(x_0)$, $q_2=f(q_1)$ and
$x_1$ one of the two points in $ p_2^{-1}(q_2)$. Then there is a (unique)
lifting $F:\s1\times \s n\to\s1\times \s n$ such that $F(x_0)=x_1$, if and
only if
\begin{equation}\label{eq:lift}
f_*{p_1}_*(\pi_1(\s1\times \s n,x_0))\subseteq
{p_2}_*(\pi_1(\s1\times\s n,x_1))\,.
\end{equation}
Assume  $F$  is such a lifting. Then obviously $\tau_2\circ F$ and
$F\circ\tau_1$ are also liftings of $f$, both distinct from $F$.
But since there are only two possible liftings of $f$, we must have
$\tau_2\circ F=F\circ\tau_1$. Hence $\tau_1$ and $\tau_2$ are conjugate.
\smallskip

Thus we are reduced to showing that we always can arrange for (\ref{eq:lift})
to hold.\par
In fact, when $\pi_1(Q_\tau)\cong \Z$ or $D_\infty$  this is immediately
true, since in these groups there is a {\em unique} infinite cyclic subgroup of
index two. \par
In the remaining case,  $\pi_1(Q_\tau)\cong \Z\oplus\Z/2$, there are two
infinite cyclic subgroups of index 2, generated by (1,0) and (1,1).
When $n$ is {\em even}, they correspond to two distinct
double covers --- $\s1\times\s n$ and $\s1\tilde\times\s n$ (mapping torus
of the antipodal map on $\s n$). Note that these are uniquely determined up
to homeomorphism by 3.3.  But any homeomorphism $f$ will lift to a
homeomorphism from $\s1\times\s n$ to {\em some} double covering of
$Q_{\tau_2}$, hence necessarily to $\s1\times \s n$.\par

When $n$ is odd, (\ref{eq:lift}) may not be a priori satisfied, but we claim
that if not, there is a homeomorphism $g$ of $Q_{\tau_2}$ such that it is
satisfied by $g f$ i.\,e. such that $g_*$ maps the two infinite cyclic
index two subgroups of $\pi_1(Q_\tau)$ to each other. \par

Let $n=2m-1$ and think of $\rp {2m-1}$ as given by complex
homogeneous coordinates $[z_1,\dots,z_m]$, and define
$g_0:\s1\times\rp {2m-1}\to\s1\times\rp {2m-1}$ by
$$g_0(t,[z_1,\dots,z_m])=(t,[\sqrt t\, z_1,\dots,\sqrt t\, z_m])\,.$$

Then ${g_0}_*$ maps the two index 2 cyclic subgroups of
$\pi_1(\s1\times\rp n)$ onto each other. Now choose a homotopy
equivalence $h:Q_{\tau_2}\to\s1\times\rp n$. $h$ determines an
element of the structure set $\mathcal S^{TOP}(\s1\times\rp n)$,
and it follows from the calculations in section 3.4 that
post--composition with $g_0$ is the identity on this set.  Thus
there exists a homeomorphism $g$ of $Q_{\tau_2}$ such that
$g_0\circ h\simeq h\circ g$. The composed homeomomorphism $gf$
now satisfies (\ref{eq:lift}).
\end{proof}
\begin{rem} Call two involutions $\tau_1$ and $\tau_2$ {\em homotopically
conjugate} if there is a homotopy equivalence 
$f:\s1\times\s n\to \s1\times\s n$ such that $f\tau_1=\tau_2 f$.
The same proof (but simpler, since the existence of a {\em homotopy 
equivalence}  $g$ is obvious in this case) shows that for $n\geqslant2$
the quotients are homotopy equivalent if and only if the involutions are 
homotopically conjugate. Hence the homotopy classification in section 2 
says that there are exactly four involutions up to homotopical
conjugacy.
\end{rem}
\smallskip

\section{Homeomorphisms of $\rp n$ up to concordance}\label{section:conc}

This section is really a digression, compared with the main objective of this
paper.  However, we believe that it is worthwile pointing out that some of our
results can be used to give a simple 
computation of the group of concordance classes of homeomorphisms of $\rp n$.
A PL version of this result (for $n\geqslant4$) is given in \cite{L}, where 
it is proved by a 
related, but more complicated argument.  Here it will be a straightforward 
consequence of some of the results of section \ref{section:surgery}.3.
\smallskip

Let $\widetilde{\Top}(RP^n)$ and  $\widetilde{\Aut}(RP^n)$ be the simplicial 
sets of block homeomorphisms and block homotopy equivalences
of $RP^n$, respectively. Recall that 
$\widetilde{\Aut}(RP^n)\simeq {\Aut}(RP^n)$.  The group of concordance 
classes of homeomorphisms of $\rp n$ is $\pi_0(\widetilde{\Top}(RP^n))$,
 which sits in the exact sequence 
\begin{multline}\label{eq:auttop}\cdots\rightarrow
%\pi_1(\widetilde{\Top}(RP^n))\rightarrow
\pi_1(\widetilde{\Aut}(RP^n))\xrightarrow{\alpha}
\pi_1(\widetilde{\Aut}(RP^n)/\widetilde{\Top}(RP^n))\rightarrow\\
\rightarrow\pi_0(\widetilde{\Top}(RP^n))\xrightarrow{\beta}
\pi_0(\widetilde{\Aut}(RP^n))\,.
\end{multline} 

$\pi_0(\widetilde{\Aut}(RP^n))\cong\pi_0({\Aut}(RP^n))$ is $\Z/2$,
generated by a reflection in an $\rp {n-1}$, for $n$ odd, and trivial for 
$n$ even . Hence $\beta$
is split surjective, so to compute $\pi_0(\widetilde{\Top}(RP^n))$ it 
suffices to compute $\text{coker}\,\alpha$.

A well--known application of the topological $s$--cobordism theorem yields
an isomorphism $\pi_1(\widetilde{\Aut}(RP^n)/\widetilde{\Top}(RP^n ))\approx
\stop(I\times\rp n)$ for $n\geqslant 4$. In fact, this is also true for 
$n=3$:  formula (\ref{eq:pnxi}) gives $\stop(I\times\rp3)\cong\Z/2$, and the
nontrivial element is represented by the homotopy self--equivalence $h_\eta$,
which also represents an element in 
$\pi_1(\widetilde{\Aut}(RP^3)/\widetilde{\Top}(RP^3 ))$. It follows that
the homomorphism $\alpha$ is equivalent to a homomorphism 
\[\alpha':\pi_1(\Aut(\rp n))\rightarrow\stop(I\times\rp n)\,,\]
and it is easy to see check that  $\alpha'$ maps an element 
of $\pi_1(\Aut(\rp n))$ to the adjoint  homotopy equivalence (rel boundary)
of $I\times\rp n$, considered as an element of  $\stop(I\times\rp n)$.
But this map has been computed in section \ref{section:surgery}.3. Only
the generator $h_\eta$ maps nontrivially, so it follows from formula
(\ref{eq:pnxi}) that
\[\ker\beta\cong\text{coker\,}\alpha\cong\text{ coker\,}\alpha'
\cong (\Z/2)^{[\frac{n-2}4]}\]

Hence we have

\begin{thm} The group of concordance classes of homeomorphisms of $\rp n$,
$n\geqslant 3$, is
\[
\pi_0(\widetilde{\Top}(\rp n))\cong
\begin{cases} (\Z/2)^{[\frac{n-2}4]}\quad when\  n\ \ is\ even\\
 (\Z/2)^{[\frac{n+2}4]}\quad when\ n\ \ is\ odd
\end{cases}
\]
\end{thm}

\begin{rem} The formula in the theorem also holds for $n=2$, since every 
self--homeomorphism of $\rp2$ is in fact {\em isotopic} to the identity.
For $n=1$ there is no element $h_\eta$, so the homomorphism $\alpha$
is trivial, and we have $\pi_0(\widetilde{\Top}(\rp 1))\approx\Z/2$, 
generated by a reflection. 
Thus we have computed $\pi_0(\widetilde{\Top}(\rp n))$ for all $n$.
\end{rem}

Note that the formula for $\pi_0(\widetilde{\Top}(\rp n))$ is the same 
as for $\pi_0(\widetilde{PL}(\rp n))$ in \cite{L}. ($n\geqslant4$.)  
In fact, we claim that the forgetful map 
\[\pi_0(\widetilde{PL}(\rp n))\to\pi_0(\widetilde{\Top}(\rp n))\] is
an isomorphism. To see this, observe that the $PL$--case can be treated 
just like that $Top$--case above, and then the claim amounts to proving
that 
\[ [\Sigma(\rp n_+), G/PL]\to [\Sigma(\rp n_+), G/Top]\]
is an isomorphism.  But this is an easy calculation. Hence we recover 
Liang's result.\smallskip

We end with the observation that from this point of view, the homeomorphism 
classes of manifolds coming from the two factors in the splitting 
$$\stop(\s1\times\rp n)\approx \stop(\rp n)\times \stop(I\times\rp n)$$
(Lemma \ref{lm:surgpn}) can be characterized by:
\smallskip

$\bullet$ The manifolds coming from $\stop(RP^n)$ are mapping
tori of the identity map, but one needs all fake projective
spaces.\par

$\bullet$ The manifolds coming from
$\stop(I\times\rp n)$ are mapping tori of the standard $\rp n$,
but one needs all concordance classes of homeomorphisms of
$\rp n$.
\smallskip

\section{Appendix. Homology with local coefficients.}

In this appendix we supply the local homology calculations necessary to
complete the proof of Lemma 2.4 for $n$ even. \par

On $\RPRP$ we let $\l$ be the local coefficient system of
groups isomorphic to $\Z$ but nontrivial over each wedge summand $\rp\infty$.
When considering spaces {\em over} $\RPRP$ we also generically let $\l$
denote the induced coefficient system. Note that $\l\otimes\l=\Z$ ---
the constant coefficient system.

\begin{ex}
$\rp n$ is a space over $\RPRP$ in two natural ways, but $\l$ is the
same in both cases. For $n$ {even} this is the orientation system on $\rp n$.
 Hence twisted Poincar\'e duality gives e.\,g.:
\begin{equation*}
 H_*(\rp{2m},\l)\cong H^{2m-*}(\rp{2m},\Z) \cong\left\{
\begin{array}{ll}
\Z & {\rm if\ }*=2m\\
\Z/2 & {\rm if\  *\  is\  even },\,0\leq * < 2m\\
0 &    {\rm otherwise}
\end{array} \right.
\end{equation*}
Just as in ordinary homology, a  $k$--connected map   $X\to Y$ over $\RPRP$
will induce isomorphisms in homology and cohomology with coefficients in $\l$
in degrees less than $k$, and if $X$ is the $k$--skeleton of $Y$, the
map on $H_k$ is onto. Hence it follows that
$$
H_*(\rp\infty,\l)\cong\left\{
\begin{array}{ll}
\Z/2 & {\rm if\  *\  is\  even}\,\geqslant 0\\
 0 &    {\rm otherwise}
\end{array}
\right.
$$
and $H_{2m}(\rp n,\l)\to H_{2m}(\rp\infty,\l)$ is surjective.
These results can, of course, also be obtained using the standard
$\Z/2$--equivariant cell structure on $\s{2m}\subset\s\infty$.

\end{ex}

A similar trick can be used to compute homology and cohomology of $\RPRP$
with coefficients in  $\l$.
Now we use that the natural map $\rpsum{2m}\to\RPRP$ is $(2m-1)$--connected
and that $\l$ is the orientation system for $\rpsum{2m}$.
Then, for $k<2m-1$ we have
$$H_k(\RPRP,\l)\cong H_k(\rpsum{2m},\l)\cong H^{2m-k}(\rpsum{2m},\Z)\,.$$

 Hence
\begin{equation*}
H_k(\RPRP,\l)\cong \left\{
\begin{array}{ll}
\Z/2 & {\rm for\ }k=0\\
\Z & {\rm for\ }k=1\\
\Z/2\oplus\Z/2 & {\rm for\ }k\ {\rm even\ }>0 \\
0 & {\rm otherwise}
\end{array}\right.
\end{equation*}

Moreover, $H_{2m}(\rp{2m},\l)\to H_{2m}(\RPRP,\l)$ is surjective.
Dually, the cohomology is
\begin{equation*}
H^k(\RPRP,\l)\cong \left\{
\begin{array}{ll}
\Z\oplus \Z/2 & {\rm for\ }k=1\\
\Z/2\oplus\Z/2 & {\rm for\ }k\ {\rm odd\ }>1 \\
0 & {\rm otherwise}
\end{array}\right.
\end{equation*}

In general, the relation between homology and cohomology is given by the
following

\begin{prop}[Universal coefficient Theorem]
For each $n$ there is a functorial, split exact sequence
$$
0\to Ext(H_{n-1}(X,\l),\Z)\to H^n(X,\l)\to Hom(H_n(X,\l),\Z)\to 0
$$
\end{prop}

\begin{proof}
Let $\pi=\pi_1(X)$ and set $CL_*=C_*(\tilde X)\otimes_{\Z\pi}\l$, where
$C_*(\tilde X)$ is the singular complex of the universal covering of $X$.
Then, by definition, $H_n(X,\l)=H_n(CL_*)$, and there is a universal
coefficient sequence
$$
0\to Ext( H_{n-1}(CL_*),\Z)\to H^n(Hom(CL_*,\Z))\to Hom(H_n(CL_*),\Z)\to 0
$$
($Hom$ means $Hom_\Z$.) $H^n(X,\l)=H_{-n}(Hom_{\Z\pi}(C_*(\tilde X), \l))$, so
to prove the proposition we only need to verify that
$Hom_{\Z\pi}(C_*(\tilde X), \l)\cong Hom(CL_*,\Z)$.
As observed above, $\l\otimes\l\cong\Z$, hence there is a canonical
isomorphism $\l\cong\l^*$. But then we have
$$
Hom_\Z(C_*(\tilde X)\otimes_{\Z\pi}\l,\Z)\cong Hom_{\Z\pi}(C_*(\tilde X),\l^*)
\cong Hom_{\Z\pi}(C_*(\tilde X),\l)\,,
$$
which is just what we want.
\end{proof}

The exact sequence  $0\to \l\xrightarrow{\cdot\, 2} \l \to \Z/2 \to 0$ gives
rise to a Bockstein homomorphism $\beta: H^k(X,\Z/2)\to H^{k+1}(X,\l)$.
Reducing coefficients mod 2 again produces a cohomology operation
$Sq_-^1: H^k(X,\Z/2)\to H^{k+1}(X,\Z/2)$.
Note that in general $Sq_-^1\ne Sq^1$. For example,
$$Sq^1:H^{k}(\rp\infty,\Z/2)\to H^{k+1}(\rp\infty,\Z/2)$$
is iso for $k$ odd and trivial for $k$ even,   but
$$Sq_-^1:H^{k}(\rp\infty,\Z/2)\to H^{k+1}(\rp\infty,\Z/2)$$
 is iso for $k$ even and trivial for $k$ odd. This follows from the Bockstein
sequence of   $0\to \l\xrightarrow{\cdot 2} \l \to \Z/2 \to 0$ and the
calculation of $H^*(\rp n,\l)$ above.\smallskip

We now have all the ingredients necessary to complete the proof of
Lemma 2.4(i) if $n$ is even. We want to use a version of diagram (2) in
section 2 were (co)homology now is taken with coefficients in $\l$ instead
of $\Z$, so we need to compute $H_n(Q,\l)$ and $H^n(Q,\l)$ and the relevant
homomorphisms.
\smallskip

First we note that $Q$ and $\RPRP$ must have isomorphic homology and
cohomology in low dimensions (less than $2m$). Since $Q$ now is orientable,
Poincar\'e duality gives
$$H_n(Q,\l)\cong H^1(Q,\l)\cong H^1(\RPRP,\l)\cong \Z\oplus\Z/2\,.$$
Hence, by the universal coefficient theorem:
$$H^n(Q,\l)\cong H_n(Q,\l)^*\cong \Z\,.$$

The Hurewicz homomorphism $h_Q:\pi_n(Q)\to H_n(Q,\l)$ is defined by
$ h_Q([f])=f_*([\s n])$.  One has to be a little careful in order to make
this definition functorial, since it involves lifting $f$ to the universal
covering of $Q$, but different choices give the same $h_Q$ up to sign, so
they don't affect the argument.

To prove that the class ${x_1}^n+{x_2}^n\in H^n(Q,\f2)$ is in the image of
$\gamma$, we use the operation $Sq^1_-$ instead of $Sq^1$. Then the proof
goes exactly as before.

\begin{rem}
One little piece of warning: $H^*(X,\l)$ does not have a product, hence there
is no ``Cartan formula'' for $Sq_-^1$.  But $H^*(X,\l)$ is
a  module over $H^*(X,\Z)$. So is $H^*(X,\Z/2)$, and the Bockstein
sequence is a sequence of $H^*(X,\Z)$--module homomorphisms. Therefore
$Sq_-^1$ is also a $H^*(X,\Z)$--module homomorphism. Note, however, that
it can not be a homomorphism of $H^*(X,\Z/2)$--modules, as the calculation
for $\rp\infty$ shows.
\end{rem}

\end{document}